\documentclass{amsart}

\usepackage{amsmath,amssymb,amsthm,amscd}

\usepackage{pinlabel}
\usepackage{enumitem}

\newtheorem{prop}{Proposition}[section]
\newtheorem{thm}[prop]{Theorem}
\newtheorem{lem}[prop]{Lemma}

\theoremstyle{definition}

\newtheorem*{defn}{Definition}
\newtheorem{ex}[prop]{Example}
\newtheorem{rem}[prop]{Remark}

\newtheorem*{ack}{Acknowledgements}

\def\co{\colon\thinspace}

\newcommand{\oalpha}{\overline{\alpha}}

\newcommand{\alphast}{\alpha_{\mathrm{st}}}

\newcommand{\C}{\mathbb{C}}
\newcommand{\CP}{\mathbb{C}\mathrm{P}}

\newcommand{\rmd}{\mathrm{d}}

\newcommand{\rme}{\mathrm{e}}

\newcommand{\PE}{\mathbb{P}E}

\newcommand{\tH}{\widetilde{H}}

\newcommand{\rmi}{\mathrm{i}}

\newcommand{\K}{\mathbb{K}}

\newcommand{\N}{\mathbb{N}}

\newcommand{\PV}{\mathbb{P}V}
\newcommand{\PP}{\mathbb{P}}

\newcommand{\Q}{\mathbb{Q}}

\newcommand{\R}{\mathbb{R}}
\newcommand{\RP}{\mathbb{R}\mathrm{P}}
\newcommand{\frakR}{\mathfrak{R}}

\newcommand{\xist}{\xi_{\mathrm{st}}}
\newcommand{\Xh}{X_{\mathrm{h}}}
\newcommand{\tX}{\widetilde{X}}

\newcommand{\Z}{\mathbb{Z}}

\DeclareMathOperator{\Int}{Int}

\begin{document}

\author[H.~Geiges]{Hansj\"org Geiges}

\address{Mathematisches Institut, Universit\"at zu K\"oln,
Weyertal 86--90, 50931 K\"oln, Germany}
\email{geiges@math.uni-koeln.de}

\title[Controlled Reeb dynamics]{Lectures on controlled Reeb dynamics}

\date{}

\dedicatory{Dedicated to David Blair on his 78th birthday}

\begin{abstract}
These are notes based on a mini-course at the conference
RIEMain in Contact, held in Cagliari, Sardinia, in June 2018.
The main theme is the connection between Reeb dynamics and topology.
Topics discussed include traps for Reeb flows, plugs for Hamiltonian
flows, the Weinstein conjecture, Reeb flows with finite numbers
of periodic orbits, and global surfaces of section for Reeb flows.
The emphasis is on methods of construction, e.g.\
contact cuts and lifting group actions in Boothby--Wang bundles, that might be
useful for other applications in contact topology.
\end{abstract}


\maketitle

\section{Introduction}
One of the driving conjectures in contact topology and Reeb dynamics
is the Weinstein conjecture about the existence of periodic Reeb orbits.
As originally envisaged by Poincar\'e in the context of the
$3$-body problem, finding periodic orbits
may be the first step (in the absence of stationary
solutions) to understanding a dynamical system:
``D'ailleurs, ce qui nous rend
ces solutions p\'eriodiques si pr\'ecieuses, c'est qu'elles sont,
pour ainsi dire, la seule br\`eche par o\`u nous puissions essayer
de p\'en\'etrer dans une place jusqu'ici r\'eput\'ee inabordable.''

The main theme of these lectures are topological constructions that
approach Reeb dynamics from the opposite direction, as it were.
The aim is to build contact manifolds whose Reeb dynamics
has certain desirable features, for instance, a given number of
periodic Reeb orbits, or a global surface of section with a
prescribed Poincar\'e return map. This is what I mean by
controlled Reeb dynamics.

In these lectures I survey some results from joint papers
with Peter Albers, Nena R\"ottgen, and Kai Zehmisch. I have
tried to put the emphasis less on specific results, but rather
on advertising the methods used to attain them. I believe that
some of the contact topological constructions I present may
turn out to be useful in other settings.

I am mostly concerned with contact topology in higher dimensions,
meaning at least five. An example of controlled Reeb dynamics
in dimension three is a paper by Colin and Honda~\cite{coho05},
where the authors construct contact forms without any contractible
periodic Reeb orbits, so-called hypertight contact forms,
on closed, orientable, irreducible, toroidal $3$-manifolds.

These notes are based essentially on three of my five lectures
in Cagliari. The other two lectures were concerned more directly
with topological aspects (surgery, cobordisms,...) of the Weinstein
conjecture. However, I felt that the material
on contact surgery is amply covered in~\cite{geig08}, and
the topics I discussed in my final lecture are well served by~\cite{geze13}.
\section{Traps and plugs in symplectic dynamics}
Starting from some simple and explicit examples of flows
on the $3$-sphere~$S^3$, I discuss the Seifert conjecture
about the existence of periodic orbits in any dynamical system
on~$S^3$. The attempts to disprove the conjecture have led to
the construction of plugs: local models of aperiodic flows that
allow one to break isolated periodic orbits in a given system.

For Reeb flows, such plugs cannot exist, but something slightly
weaker, what we call traps, does. I then cast my net a little
wider and include Hamiltonian flows, for which plugs can in fact be
constructed. I explain one such plug that is built from a Reeb trap.
\subsection{Flows on the $3$-sphere}
\label{subsection:flowS3}
Think of the $3$-sphere as the union of $\R^3$ with a point at
infinity, $S^3=\R^3\cup\{\infty\}$. In turn, we visualise $\R^3$ as
being obtained by rotating the drawing plane about a vertical
axis. Thus, the two points shown in Figure~\ref{figure:hopf},
symmetric with respect to the axis of rotation, actually represent
a circle. The axis of rotation, as it passes through the point
at infinity, likewise constitutes a circle in $S^3$.
Each pair of circles symmetric with respect to the axis
represents a $2$-torus. Figure~\ref{figure:hopf} therefore illustrates
how to decompose $S^3$ into two circles and an infinite family of $2$-tori.
If we regard $S^3$ as the unit sphere in $\C^2$,
\[ S^3=\bigl\{(z_1,z_2)\in\C^2\co |z_1|^2+|z_2|^2=1\bigr\},\]
these so-called Hopf tori are defined by
\[ T^2_r=\bigl\{ (z_1,z_2)\in S^3\co |z_1|=r\bigr\},\;\;r\in ]0,1[;\]
the two circles $C_1,C_2$ are defined by $C_j=S^3\cap\{z_j=0\}$.

\begin{figure}[h]
\labellist
\small\hair 2pt
\pinlabel $C_2$ [bl] at 425 398
\pinlabel $T^2$ [bl] at 577 408
\pinlabel $C_1$ [tl] at 546 54
\endlabellist
\centering
\includegraphics[scale=0.45]{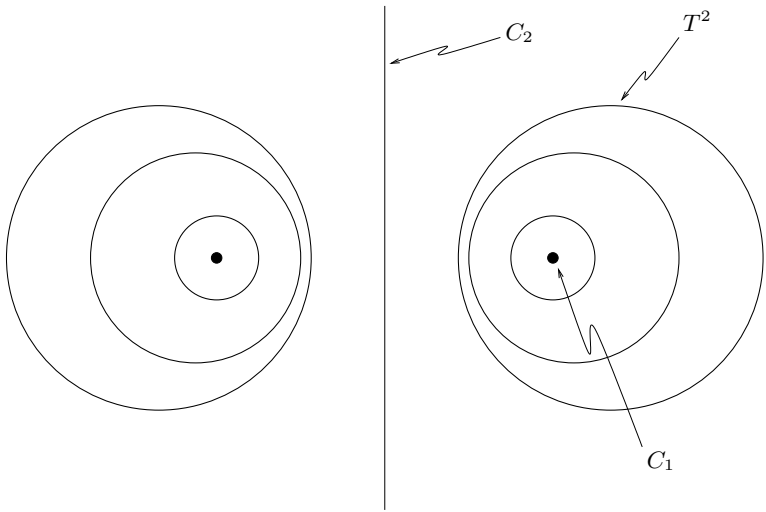}
  \caption{Decomposition of $S^3$ into Hopf tori and two circles.}
  \label{figure:hopf}
\end{figure}

Figure~\ref{figure:torus-circles} illustrates how to foliate
each Hopf torus by circles going once around each $S^1$-factor
in $T^2=S^1\times S^1$. Notice that as $r\rightarrow 0$ or
$r\rightarrow 1$, these $(1,1)$-circles approach the
two circles $C_1,C_2$. This defines a foliation of $S^3$ by circles
where any two circles form a Hopf link, see Figure~\ref{figure:hopf-link}.
Observe that each of these circles, with the exception of~$C_2$,
intersects the open half-plane on the right of Figure~\ref{figure:hopf}
in a single point. This half-plane, together with $C_2$,
constitutes a closed $2$-disc. Thus, the leaf space of this foliation
(or the orbit space, if we think of the circles as flow lines)
is a $2$-disc with its boundary collapsed to a point, in other words:
a $2$-sphere.

\begin{figure}[h]
\centering
\includegraphics[scale=0.25]{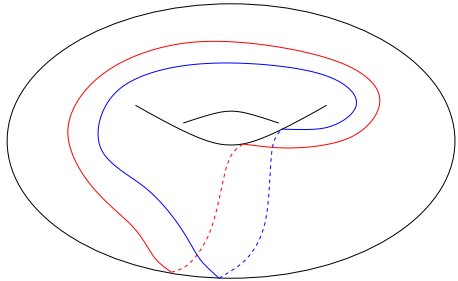}
  \caption{Foliation of $T^2$ by $(1,1)$-circles.}
  \label{figure:torus-circles}
\end{figure}

\begin{figure}[h]
\centering
\includegraphics[scale=0.3]{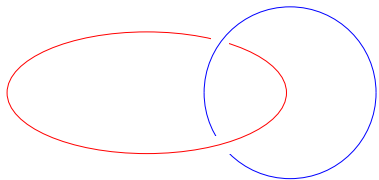}
  \caption{A Hopf link in $S^3$.}
  \label{figure:hopf-link}
\end{figure}

What we have described here is the topological visualisation
of the foliation given by the fibres of the Hopf fibration
\[ \begin{array}{ccc}
\C^2\supset S^3 & \longrightarrow & \CP^1=S^2\\
(z_1,z_2)       & \longmapsto     & [z_1:z_2].
\end{array}\]
More simply (and dynamically), we can think of the circles in this
foliation as the orbits of the vector field $\partial_{\varphi_1}+
\partial_{\varphi_2}$, where $\varphi_j$ denotes the angular
coordinate in the $z_j$-plane.

When we perturb the $(1,1)$-foliation on each Hopf torus
into a foliation by lines of irrational slope (close to~$1$),
all leaves of the foliation but the two Hopf circles $C_1,C_2$
open up. Analytically, this corresponds to passing to the
vector field $\partial_{\varphi_1}+(1+\varepsilon)\partial_{\varphi_2}$
with $\varepsilon\in\R\setminus\Q$ close to~$0$. The only orbits of this
vector field that close up are the $C_j=S^3\cap\{z_j=0\}$, $j=1,2$.

Can we find a flow on $S^3$ with only a single periodic orbit?
Figure~\ref{figure:index2} shows a flow on the $2$-sphere with
a single fixed point of index~$2$. Using the Hopf fibration,
we can lift the vector field defining this flow to
a vector field on $S^3$ orthogonal to the Hopf vector
field $\partial_{\varphi_1}+\partial_{\varphi_2}$. The sum of these two
vector fields then defines a flow with only one periodic orbit: the
Hopf fibre over the singularity on~$S^2$.

\begin{figure}[h]
\centering
\includegraphics[scale=0.25]{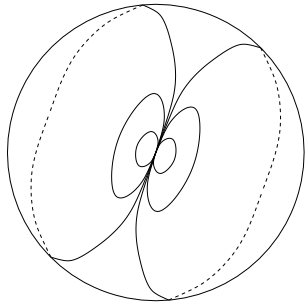}
  \caption{A flow on $S^2$ with a single index $2$ singularity.}
  \label{figure:index2}
\end{figure}

Does every non-singular flow on $S^3$ have a periodic orbit?
This question was posed by Seifert~\cite{seif50} in 1950, and the positive
answer to the question became known as the Seifert conjecture, even though
Seifert did not commit himself either way. For vector fields of class~$C^1$,
the Seifert conjecture was disproved by Schweitzer~\cite{schw74} in 1974.
In 1994, K.~Kuperberg~\cite{kupe-k94} constructed a non-singular vector
field of class $C^{\infty}$ on $S^3$ without any periodic orbits.
Her construction used a modification of what is known
as Wilson's \emph{plug}~\cite{wils66}, a concept I am going to
describe next.
\subsection{Traps and plugs}
By the flow box theorem (or tubular flow theorem),
see~\cite[Theorem~2.1.1]{pame82}, the flow of any non-singular vector field
on an $m$-dimensional manifold locally looks like $(x,s)\mapsto (x,s+t)$
on $D^{m-1}\times[0,1]$, i.e.\ the flow lines are $\{x\}\times [0,1]$.
Traps and plugs are local models of flows that can be inserted in
place of such a flow box.

\begin{defn}
A \emph{trap} (Figure~\ref{figure:trap})
is a non-singular flow on $D^{m-1}\times [0,1]$ with
the following properties:
\begin{enumerate}[widest=(iii)]
\item[(i)] the flow is aperiodic, i.e.\ there are no periodic orbits;
\item[(ii)] there is an orbit entering at $D^{m-1}\times\{0\}$
that does not leave;
\item[(iii)] the flow is parallel to $[0,1]$ along the boundary of
$D^{m-1}\times [0,1]$.
\end{enumerate}
\end{defn}

The trap contains an aperiodic invariant set; an orbit asymptotic
in forward time to this invariant set is trapped.

\begin{figure}[h]
\centering
\includegraphics[scale=0.4]{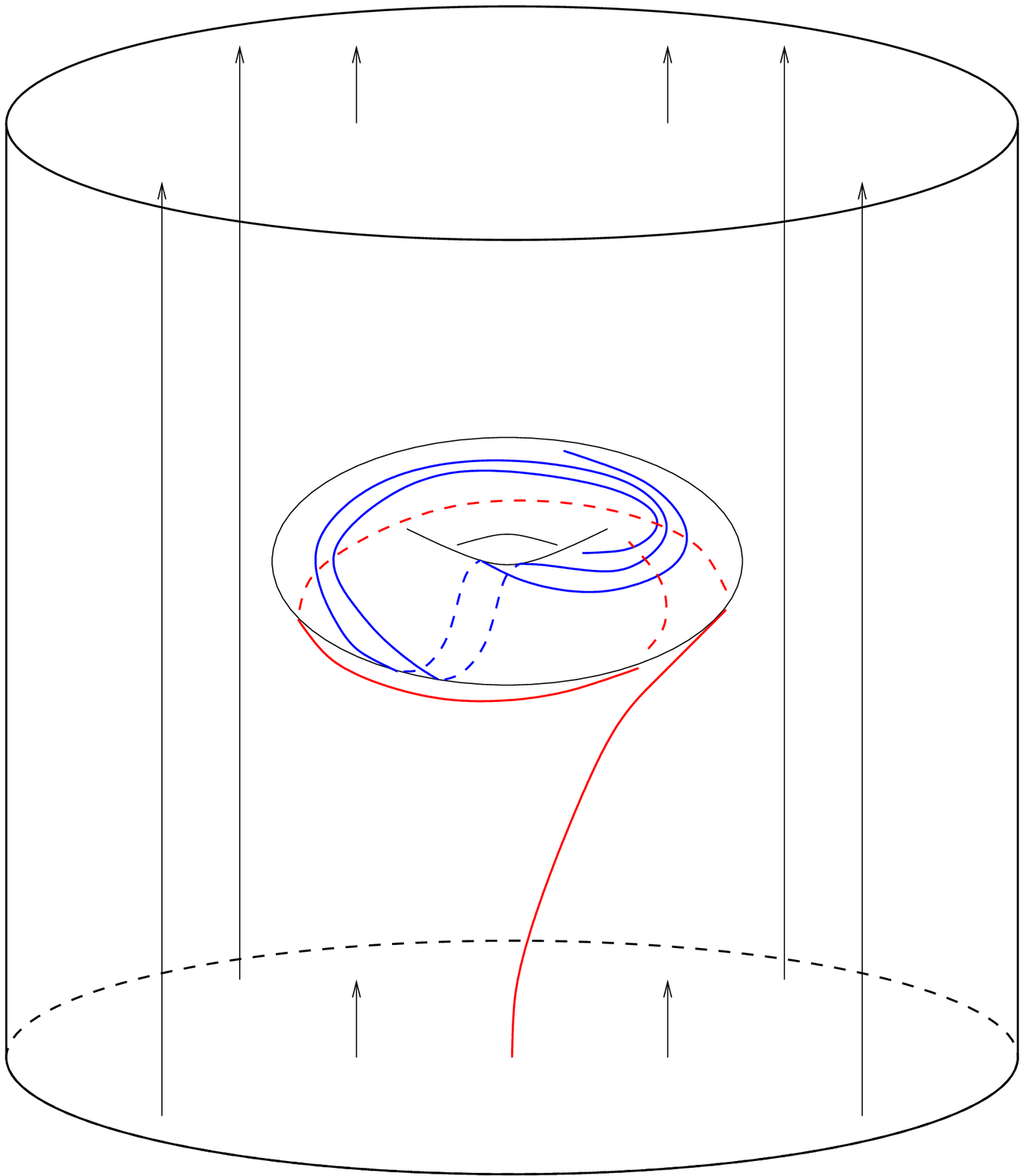}
  \caption{A trap.}
  \label{figure:trap}
\end{figure}

Inserting a trap in place of the original flow box allows one
to open up an isolated periodic orbit in the original flow.
However, an orbit entering at $(x,0)$ and passing through the trap
will in general exit at some point $(x',1)$ with $x'\neq x$.
This means that we lose control over the global dynamics.

\begin{defn}
A \emph{plug} is a trap with matching condition: for any $x\in D^{m-1}$,
the orbit entering at $(x,0)$ is either trapped or exits at $(x,1)$. 
\end{defn}

By putting two traps in sequence, one being the mirror image of the
other with reversed flow direction, one can create a plug,
see Figure~\ref{figure:trap-double}. Of course, there may well
be plugs that do not come from doubling a trap.

\begin{figure}[h]
\centering
\includegraphics[scale=0.4]{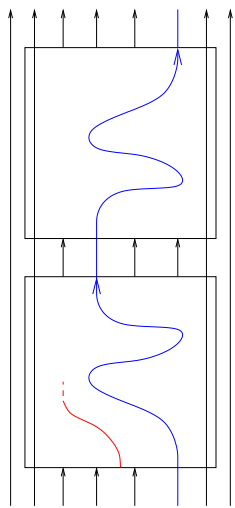}
  \caption{Doubling a trap yields a plug.}
  \label{figure:trap-double}
\end{figure}

Inserting a plug in place of
a flow box will not affect the behaviour of the orbits
traversing the plug. Thus, with the help of a plug, any
dynamical system with isolated periodic orbits can be turned
into a flow without any periodic orbits.

Finding a trap or a plug should become easier with increasing
dimension. For instance, one could imagine taking a $2$-torus
with an irrational flow as aperiodic invariant set inside the
trap. In three dimensions, however, this would force the
existence of a periodic orbit in the interior of the $2$-torus.

Also, if one is interested in flows preserving some
geometric structure, it may well be possible to find a trap,
but doubling it may be impossible since this involves
reversing the flow direction, which the geometric structure
might obstruct.

Here is a list of plugs for certain geometric flows. The notion of
Hamiltonian and Reeb flows will be introduced presently. The
smoothness class refers to that of the vector field defining the flow.
See also the survey~\cite{kupe-k99} and the introduction to~\cite{grz16}.

\begin{itemize}
\item Wilson~\cite{wils66}: volume-preserving, $C^{\infty}$, $\dim\geq 4$ 
\item G.~Kuperberg~\cite{kupe-g96}: volume-preserving, $C^1$,
$\dim =3$.
\item Ginzburg~\cite{ginz97}, Herman~\cite{herm99}, Kerman~\cite{kerm02},
Geiges--R\"ottgen--Zehmisch~\cite{grz16}: Hamiltonian, $C^{\infty}$,
$\dim\geq 5$ (here the dimension refers to that of the energy hypersurface).
\item Ginzburg--G\"urel~\cite{gigu03}: Hamiltonian, $C^1$, $\dim=3$.
\end{itemize}

Wilson's plug and the Hamiltonian plug constructed in \cite{grz16}
involve a doubling construction. I shall present a few more details
of the latter in Section~\ref{subsection:trap-plug}.
\subsection{Hamiltonian and Reeb flows}
We have the following hierarchy of geometric flows:
\begin{eqnarray*}
\lefteqn{\text{volume-preserving} \supset \text{symplectic}\supset
\text{Hamiltonian} \supset  \text{Reeb} \supset}\\
 & \supset &  \text{Finsler geodesic} \supset \text{Riemannian geodesic}.
\end{eqnarray*}
The first three inclusions I shall explain presently, for the last two
see~\cite{dgz17} and~\cite[Section~1.5]{geig08}. This is one
potential motivation for the study of Hamiltonian and Reeb flows.
For instance, a statement about the existence of periodic
orbits for Reeb flows on unit cotangent bundles~\cite{hovi88}
is, \emph{a fortiori}, a statement about
closed geodesics in Finsler or Riemannian geometry~\cite{lyfe51}.
Another example in this vein is~\cite{fls15}.

\subsubsection{Hamiltonian flows}
Let $(W,\omega)$ be a $2n$-dimensional \emph{symplectic manifold},
that is, $\omega$ is a closed, non-degenerate $2$-form on~$W$.
Given a smooth function $H\co W\rightarrow\R$, the
\emph{Hamiltonian vector field} $X_H$ is defined by
\[ \omega(X_H,\,.\,)=\rmd H.\]

\begin{rem}
The letter $H$ actually stands for `Huygens', see~\cite[p.~187]{geig16}.
\end{rem}

Observe that $\rmd H(X_H)=0$, so the flow of $X_H$ is tangent
to the level sets of~$H$. In classical mechanics, $H$ is typically
the total energy, and the Hamiltonian flow describes
the dynamics of the system. Then the statement $\rmd H(X_H)=0$ means
conservation of energy.

Moreover, the Lie derivative of $\omega$ in the direction
of $X_H$ is, by Cartan's formula,
\[ L_{X_H}\omega=\rmd(i_{X_H}\omega)+i_{X_H}\rmd\omega=\rmd^2H=0.\]
In particular, the Hamiltonian flow preserves the volume form $\omega^n$
on what, in classical mechanics, is the phase space of the system;
this is known as Liouville's theorem. This explains the first two
inclusions above.

\begin{ex}
Not every symplectic flow is Hamiltonian. For instance, on
a closed manifold the Hamiltonian vector field necessarily has zeros.
So the flow of $\partial_{\varphi_1}$ on $(T^2=S^1\times S^1,\rmd\varphi_1
\wedge\rmd\varphi_2)$, which preserves the symplectic form, cannot
be Hamiltonian.
\end{ex}
\subsubsection{Reeb flows}
Let $(M,\xi=\ker\alpha)$ be a $(2n-1)$-dimensional \emph{contact manifold},
that is, $\xi$ is a hyperplane field (which I always assume to be
coorientable), whose defining $1$-form~$\alpha$ satisfies the condition
$\alpha\wedge(\rmd\alpha)^{n-1}\neq 0$. For a given~$\xi$, this condition
is independent of the choice of~$\alpha$. The hyperplane field $\xi$ is called
a \emph{contact structure}; $\alpha$ is called a
\emph{contact form} for~$\xi$.

The \emph{Reeb vector field} $R=R_{\alpha}$ of $\alpha$ is defined uniquely by
the conditions
\[ i_R\rmd\alpha=0\;\;\;\text{and}\;\;\;\alpha(R)=1.\]

\begin{ex}
\label{ex:xist}
The $1$-form
\[ \alpha_{\varepsilon}=x_1\,\rmd y_1-y_1\,\rmd x_1+\frac{1}{1+\varepsilon}\,
(x_2\,\rmd y_2-y_2\,\rmd x_2)\]
is a contact form on $S^3\subset\R^4$ for $\varepsilon\neq -1$.
Its Reeb vector field is $R=\partial_{\varphi_1}+(1+\varepsilon)
\partial_{\varphi_2}$ --- the vector field we encountered in
Section~\ref{subsection:flowS3}.
\end{ex}

The Reeb flow on $(M,\alpha)$ equals the Hamiltonian flow
on $H^{-1}(0)$ for $W=\R\times M$, $\omega=\rmd(\rme^t\alpha)$
and $H(p,t)=\rme^t$. This proves the third inclusion above.
The Reeb flow in the example may also be interpreted as the
Hamiltonian flow on an ellipsoid in $\R^4$ with its
standard symplectic form $\rmd x_1\wedge\rmd y_1+\rmd x_2\wedge\rmd y_2$.

\begin{rem}
The Reeb flow on $(M^{2n-1},\alpha)$ is volume-preserving for
the volume form $\alpha\wedge(\rmd\alpha)^{n-1}$.
\end{rem}
\subsection{The Weinstein conjecture}
\label{subsection:WC}
The Weinstein conjecture~\cite{wein79} asserts that any
Reeb flow on a closed manifold has a periodic orbit.
A combination of results of Rabinowitz~\cite{rabi79},
Eliashberg~\cite{elia89,elia92} and Hofer~\cite{hofe93} settles
this conjecture for the $3$-sphere. Eliashberg establishes a dichotomy
between so-called tight and overtwisted contact structures on $3$-manifolds.
On the $3$-sphere, there is a unique tight contact structure up
to isotopy: the standard contact structure $\xist=\ker\alpha_0$
described in Example~\ref{ex:xist}. Overtwisted contact structures
are determined by the homotopy class of the underlying
tangent $2$-plane field; on $S^3$, there is an integer family of
such structures, classified by the Hopf invariant.

\begin{rem}
Notice that a contact structure $\xi=\ker\alpha$ on a $3$-manifold
determines an orientation: the sign of the volume form $\alpha\wedge
\rmd\alpha$ does not depend on the choice of $\alpha$ defining a given~$\xi$.
All statements about classification of contact structures on $S^3$
here refer to \emph{positive} contact structures, i.e.\ the ones
inducing the standard orientation of $S^3\subset\R^4$.
\end{rem}

The contact forms on $S^3$ defining $\xist$, with the same coorientation
as the one given by~$\alpha_0$, are the $f\alpha_0$ with
$f\co S^3\rightarrow\R^+$. This is the same as restricting the
$1$-form $\alpha_0$ on $\R^4$ to the starshaped hypersurface
$fS^3$, for which the Weinstein conjecture was established by Rabinowitz.
For the overtwisted contact structures on $S^3$ (or any other
closed, orientable $3$-manifold), the conjecture follows from
Hofer's result.

In other words, the Seifert conjecture holds for Reeb vector fields.

For all closed, orientable $3$-manifolds, the Weinstein conjecture
was proved by Taubes~\cite{taub07}. Cristofaro-Gardiner and
Hutchings~\cite{cghu16} improved this by showing that every
Reeb flow on a $3$-manifold has at least two periodic orbits.
Together with Pomerleano~\cite{chp17} they showed that
if the first Chern class of the contact structure is torsion,
and the contact form is non-degenerate (meaning that the
return map near the periodic Reeb orbits never has $1$ as
an eigenvalue), there are either two or infinitely many
periodic Reeb orbits. The case of `two' only arises on lens spaces.
This dichotomy between two or infinitely many periodic Reeb orbits
is illustrated by Example~\ref{ex:xist}.

In particular, these results show that there can be no plugs for Reeb
vector fields in dimension three. In higher dimensions, one can argue
similarly --- using various instances where the Weinstein
conjecture holds --- to prove the non-existence of Reeb plugs.
However, as I want to show, there are
Reeb traps in dimensions at least five. For this we need to introduce the
concept of contact Hamiltonians.

\begin{rem}
\label{rem:rukimbira}
At the Cagliari conference I learned from David Blair
about earlier results in metric contact geometry concerning
the minimal number of periodic Reeb orbits, see~\cite[Section~3.4]{blai10}.
Rukimbira~\cite{ruki95} has shown that on a $(2n-1)$-dimensional closed
$K$-contact manifold (a metric contact manifold whose Reeb vector
field is Killing), there are at least $n$ periodic Reeb orbits.
If the manifold is simply connected and there are precisely
$n$ periodic orbits, the manifold is homeomorphic to a sphere~\cite{ruki99}.
\end{rem}

The number $n$ of periodic Reeb orbits on a $(2n-1)$-dimensional
contact manifold is realised on any irrational ellipsoid,
generalising Example~\ref{ex:xist}.
See also the construction in Section~\ref{section:finite}. One
may well conjecture this to be the minimal number of periodic Reeb
orbits in general.
\subsection{Contact Hamiltonians}
A \emph{contact vector field} $X$ on a contact manifold $(M,\xi)$
is a vector field whose flow preserves the contact structure~$\xi$.
When we choose a contact form $\alpha$ that defines $\xi=\ker\alpha$,
the condition on $X$ becomes $L_X\alpha=\lambda\alpha$ for some
function $\lambda\co M\rightarrow\R$.

A choice of contact form $\alpha$ for $\xi$ sets up a one-to-one
correspondence between contact vector fields $X$ and smooth functions
on $H\co M\rightarrow\R$ as follows. Given $X$, set $H_X:=\alpha(X)$.
Conversely, given $H$, define
\[ X_H:=HR+Y,\]
where $R$ is the Reeb vector field of~$\alpha$, and the
vector field $Y\in\xi$ is defined by
\[ i_Y\rmd\alpha=\rmd H(R)\alpha-\rmd H.\]
Since $\alpha(X_H)=H$, we have $H_{X_H}=H$. I leave it to the
reader to check that $X_{H_X}=X$, or see~\cite{geig08}.
The computation
\[ L_{X_H}\alpha=\rmd(i_{X_H}\alpha)+i_{X_H}\rmd\alpha
=\rmd H+i_{Y}\rmd\alpha=\rmd H(R)\alpha\]
shows that $X_H$ is indeed a contact vector field.

\begin{ex}
\label{ex:rescale}
The Reeb vector field corresponds to the constant function~$1$.
It follows that if $H>0$, then $X_H$ is the Reeb vector field of the
contact form $\alpha/H$, since $X_H$ preserves $\ker(\alpha/H)=\ker\alpha$,
and $(\alpha/H)(X_H)=1$.
\end{ex}

\begin{rem}
Observe that in contrast with the symplectic case, all
contact vector fields are (contact) Hamiltonian.
\end{rem}
\subsection{Reeb traps and Hamiltonian plugs}
\label{subsection:trap-plug}
\subsubsection{Non-existence of Reeb traps in dimension three}
The non-existence of Reeb traps in dimension three is a consequence
of the following global Darboux theorem due to
Eliashberg and Hofer~\cite{elho94}.

\begin{thm}[Eliashberg--Hofer]
\label{thm:elho}
Let $\alpha$ be a contact form on $\R^3$ that coincides with the
standard form $\alphast=\rmd z + (x\,\rmd y-y\,\rmd x)/2$ outside a compact
set. If the Reeb vector field $R_{\alpha}$ of $\alpha$ does not have any
periodic orbit, then $(\R^3,\alpha)$ is diffeomorphic to $(\R^3,\alphast)$.
\end{thm}

Notice that a diffeomorphism that sends one contact form to
the other (a so-called \emph{strict} contactomorphism),
also maps the Reeb vector field of one to the other.
This follows immediately from the defining equations of the
Reeb vector field.

Now, the Reeb vector field of $\alphast$ is $\partial_z$, which does not
have any trapped Reeb orbits. Thus, if $R_{\alpha}$ has a trapped
Reeb orbit, $(\R^3,\alpha)$ cannot be diffeomorphic to
$(\R^3,\alphast)$. Therefore, $R_{\alpha}$ must then also have
a periodic orbit, or else this would contradict the theorem.
So a trap, which is required to be aperiodic, cannot exist.

The idea of the proof of Theorem~\ref{thm:elho} is roughly as
follows. One studies the moduli space of holomorphic discs in
the symplectisation $\bigl(\R\times\R^3,\rmd(\rme^t\alpha)\bigr)$,
with boundary on a cylinder $\{0\}\times Z$ containing the region
in $\{0\}\times\R^3$ where $\alpha$
differs from~$\alphast$. The almost complex structure on
the symplectisation is one that preserves $\ker\alpha$ and sends
$\partial_t$ to $R_{\alpha}$. In particular, cylinders
over periodic Reeb orbits are holomorphic curves.

If this moduli space is non-compact, this leads to the breaking
of holomorphic curves along ends that become asymptotic to
such cylinders, which necessitates the existence of
periodic Reeb orbits.

If, on the other hand, this moduli space is compact, it leads to
a filling of the cylinder $\{0\}\times Z\subset\R^4$
by holomorphic discs, which descends
to a filling by discs of $Z\subset\R^3$ (this latter conclusion
is not obvious). By construction, the Reeb vector field is transverse
to these discs, which prevents the existence of a trapped orbit.

This proves the corollary about the non-existence of Reeb traps,
and by using the Reeb flow to define a coordinate (in the second
alternative), one obtains the theorem.
\subsubsection{Existence of Reeb traps in higher dimensions}
By contrast, we have the following result~\cite{grz14}.

\begin{thm}[Geiges--R\"ottgen--Zehmisch]
There are Reeb traps in all odd dimensions~$\geq 5$.
\end{thm}

I describe the idea of the proof in dimension five; adapting this
proof to the general case is just a matter of notation.
We would like to construct a contact vector field $X$ for the standard
contact structure $\ker\alphast$, where
\[ \alphast=\rmd z+
\frac{1}{2}\sum_{j=1}^2(x_j\,\rmd y_j-y_j\,\rmd x_j),\]
such that

\begin{enumerate}[widest=(iii)]
\item[(i)] $X=\partial_{\varphi_1}+s\partial_{\varphi_2}$, $s\in [0,1]
\setminus\Q$ on the Clifford torus $\{r_1=r_2=1,\, z=0\}$;
\item[(ii)] $T\times [-1,0]$ is mapped to itself under the positive
$X$-flow;
\item[(iii)] $X=\partial_z$ outside a compact set;
\item[(iv)] $\rmd z(X)>0$ on $\R^5\setminus T$.
\end{enumerate}

These conditions translate into properties of the corresponding
Hamiltonian function~$H$, and one can show that a function $H>0$
with these properties exist. By Example~\ref{ex:rescale}, $X$
is the Reeb vector field of the rescaled contact form
$\alphast/H$.

The Clifford torus with the irrational foliation serves as the
aperiodic invariant set, which traps orbits by~(ii). Condition (iii)
guarantees that we only change the Reeb flow in a compact set.
Condition (iv) ensures that the flow is aperiodic.
\subsubsection{From a Reeb trap to a Hamiltonian plug}
A Hamiltonian plug can be constructed by doubling this
Reeb trap, see~\cite{grz16}. This simplifies earlier constructions of
Hamiltonian plugs.

Place one Reeb trap in the half-space $\{z<0\}$, and
put a mirror image of it in the half-space $\{ z>0\}$
by pulling it back via $\Phi\co z\mapsto -z$. On this mirror image
we need to work with the reversed Reeb flow, i.e.\ the
flow of $-\Phi^*R_{\alpha}$. Here the attempt to build a Reeb plug
breaks down, but as a Hamiltonian plug this works just fine.

The reason is essentially that the vector field $\partial_z$ is
the Hamiltonian vector field both of $(\R^{2n+1},\rmd\alphast)$
(in standard symplectic $\R^{2n+2}$) and for
$(\R^{2n+1},\rmd\Phi^*\alphast)$, for the same coorientation.
Replacing $\alphast$ by $\alphast/H$, the contact form used
as a trap, amounts to a (compactly supported) deformation of
$\R^{2n+1}$ in $\R^{2n+2}$.
\section{Cuts}
Contact cuts are a topological method introduced by
Lerman~\cite{lerm01} for constructing contact manifolds. The method
has topological and symplectic predecessors. The notion of `cut'
probably first arose as an alternative description
of blow-up constructions, and this is the view I take here.
The main advantage of this method over more flexible
topological gluings, say, is that it allows one
explicit control over the contact form, since no interpolation
of differential forms over gluing regions is required.
An application of contact cuts will be presented in
Section~\ref{section:embedding}.
\subsection{Blowing up}
(i) Let $V$ be a vector space over a field $\K$. Its projectivisation
$\PV$ is the space of all one-dimensional subspaces $\ell\subset V$ or,
equivalently, the quotient of $V\setminus\{0\}$ under the equivalence
relation
\[ x\sim y\;\; :\Longleftrightarrow\;\;
\text{$x=\lambda y$ for some $\lambda\in\K$.}\]
In the real or complex case we may alternatively think of
$\PV$ as the quotient of the unit sphere in $V$ under the action
of the elements $\lambda\in\K$ of unit length, i.e. $\Z_2$ or $S^1$,
respectively.

Over $\PV$ there is a tautological line bundle~$\eta(V)$: over the `point'
$\ell\in\PV$ we have the line made up of all the points
$x\in\ell\subset V$. An explicit realisation of $\eta(V)$ is given by
\[ \eta(V):=\{(\ell,x)\in\PV\times V\co x\in\ell\},\]
with bundle projection $(\ell,x)\mapsto\ell$.
We identify $\PV$ with the zero section of $\eta(V)$.

Observe that we have a canonical identification
\[ \begin{array}{ccc}
\eta(V)\setminus\PV & \stackrel{\cong}{\longrightarrow} & V\setminus\{0\}\\
(\ell,x)            & \longmapsto                      & x,
\end{array}\]
since any point $x\in V\setminus\{0\}$ determines a unique $1$-dimensional
subspace.

(ii) This construction easily generalises to vector bundles $E\rightarrow Q$.
Such a bundle admits a projectivisation $\PE\rightarrow Q$, and over this
projectivised bundle we have a tautological line bundle
$\K\hookrightarrow\eta(E)\rightarrow\PE$. As before, we can canonically
identify $\eta(E)\setminus PE$ with $E\setminus Q$.

(iii) The \emph{blow-up} of a differential (real or complex) manifold $M$
along a (real or complex) submanifold $Q$ is now defined as follows.
Identify an open tubular neighbourhood  of $Q$ in $M$ with
the total space $\nu Q$ of the normal bundle of~$Q$. Then form the
quotient space
\[ (M\setminus Q)\cup \eta(\nu Q)/\!\sim\]
under the identification
\[ \nu Q\setminus Q\ni x\sim([x],x)\in\eta(\nu Q)\setminus\PP(\nu Q).\]
The resulting space is again a manifold. The effect of the
construction is to replace the submanifold $Q$ by the projectivisation
of its normal bundle.

\begin{ex}
The tautological line bundle over $\PP(\R^2)\cong S^1$ is a M\"obius band
(without its boundary points). Hence, blowing up a point $q$ in a real
$2$-dimensional manifold $M$ is the same as cutting out a disc around $q$
and gluing in a M\"obius band under the identification
\[ M\setminus\Int(D^2)\cong S^1\cong\partial(\text{M\"obius band}).\]
The effect is to replace $q$ by the spine $S^1$ of the M\"obius band.
The M\"obius band may be thought of as
\[ \{ (t\cos\theta,t\sin\theta,[\theta])\in\R\times\R\times\R/\pi\Z\co
|t|\leq 1\}.\]
The projection
\[ (t\cos\theta,t\sin\theta,[\theta])\longmapsto
(t\cos\theta,t\sin\theta)\]
sends the spine of the M\"obius band to~$0$; the complement of the
spine is mapped diffeomorphically onto $D^2\setminus\{0\}$,
see Figure~\ref{figure:blowup}.
Beware that $t$ is not a global coordinate on the M\"obius band,
which amounts to saying that the line bundle $\eta(\R^2)\rightarrow S^1$
is non-trivial.
\end{ex}

\begin{figure}[h]
\centering
\includegraphics[scale=0.4]{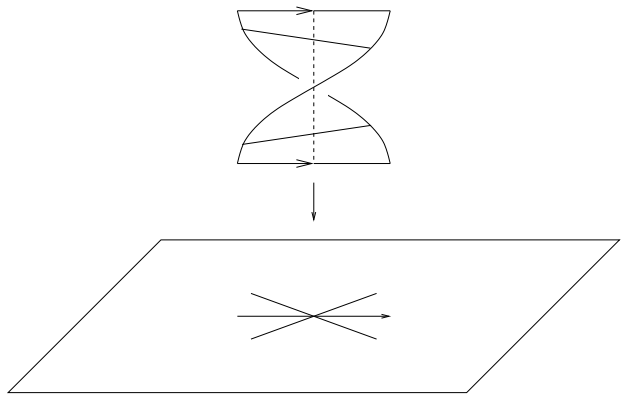}
  \caption{Blowing up a point in $\R^2$.}
  \label{figure:blowup}
\end{figure}

\begin{rem}
For the complex blow-up, the essential ingredient is a complex
bundle structure on the normal bundle~$\nu Q$. This is the key
to generalising the blow-up construction to symplectic submanifolds
of a symplectic manifold, see~\cite{mcdu84}. There are a number of
subtleties here, such as the question of uniqueness of the
symplectic form on the blow-up, see~\cite[Section~7.1]{mcsa17}.
Nonetheless, one can make sense of the Chern classes of the
blown-up symplectic manifold, see~\cite{gepa07}.
\end{rem}

\begin{rem}
A M\"obius band is the same as the complement of an open disc
in the real projective plane~$\RP^2$. Thus, blowing up a point
in a surface is the same as forming the connected sum with
a copy of~$\RP^2$. In dimension~$n$, analogously, blowing
up a point is topologically the same as a connected sum
with~$\RP^n$. In the complex case, it amounts to
a connected sum with a copy of $\overline{\CP}^n$,
a complex projective space with the opposite of
its natural orientation. We shall make use of this kind of blow-up
in Section~\ref{section:Ham-finite}.
\end{rem}

(iv) An alternative view of the blow-up construction is the
following. Remove the tubular neighbourhood $\nu Q\subset M$, and
identify points on the boundary sphere $\partial (M\setminus\nu Q)$
under the $\Z_2$-action or the $S^1$-action that defines
the projectivisation in the real or complex case, respectively.
The boundary, under this quotient, becomes $\PP(\nu Q)$.

Though it may not be immediately apparent that this
construction even produces a smooth manifold,
it is not difficult to see that it is indeed equivalent to the
previous description of a blow-up.

Often, this alternative viewpoint is the more appropriate one.
A good example are blow-ups of symplectic manifolds. When one
removes an open ball in a Darboux chart, the symplectic form on
the boundary sphere degenerates along the Hopf fibration.
Collapsing the $S^1$-fibres of the Hopf fibration produces
a symplectic quotient manifold~\cite[Section~7.1]{mcsa17}, the
symplectic blow-up of a point. From this point of view,
for instance, it is obvious why blowing up decreases the symplectic volume.

The notion of `cuts' \cite{lerm95,lerm01} provides the language to make these
statements precise.
\subsection{Topological cuts}
Suppose we are given a smooth $S^1$-action on a manifold $M$,
which we write as $p\mapsto\rme^{\rmi\varphi}p$ for $p\in M$
and $\rme^{\rmi\varphi}\in S^1$. Assume further that we have
an $S^1$-invariant function $f\co M\rightarrow\R$ with a
regular level set $f^{-1}(0)$ on which $S^1$ acts freely.
The smooth manifold $M'$ constructed in the next proposition is called
the \emph{cut} with respect to the given data.

\begin{prop}
\label{prop:cut-top}
The quotient space
\[ M':=\{ p\in M\co f(p)\geq 0\}/\text{$p\sim\rme^{\rmi\varphi}p$
for $f(p)=0$ and $\rme^{\rmi\varphi}\in S^1$}\]
is a smooth manifold.
\end{prop}

\begin{proof}
The idea of the proof is to identify $M'$
with the quotient of a larger manifold under a \emph{free}
$S^1$-action. Indeed, $S^1$ acts on the product manifold
$M\times\C$ by the anti-diagonal action
\[ (p,z)\longmapsto (\rme^{\rmi\varphi}p,\rme^{-\rmi\varphi}z).\]
The function
\[ F(p,z):=f(p)-|z|^2\]
is $S^1$-invariant, and from
\[ \rmd F=\rmd f-2r\,\rmd r\]
we see that $F^{-1}(0)$ is a regular level set.
The $S^1$-action on this level set is free: for $f(p)=0$, the action is free
on the $M$-factor; for $f(p)>0$, on the $\C$-factor. It follows that
the quotient $F^{-1}(0)/S^1$, which coincides with $M'$,
is a smooth manifold.
\end{proof}

\begin{rem}
To see the identification of the two quotients $M'$ and
$F^{-1}(0)/S^1$ geometrically,
observe the following. When we restrict attention to
a collar neighbourhood $[0,\varepsilon)\times f^{-1}(0)$
of $f^{-1}(0)$ in $\{f\geq 0\}$, the level set $F^{-1}(0)$
looks like the product of $f^{-1}(0)$ with a standard paraboloid
in $\C\times\R$, on which $S^1$ acts by rotation.
\end{rem}

In order to define the quotient space $M'$ in Proposition~\ref{prop:cut-top},
it would be sufficient to have the $S^1$-action defined
on $f^{-1}(0)$. However, in order to argue as in the proof and
ensure that the quotient is smooth, one needs the $S^1$-action
to be defined at least in a collar neighbourhood. This, of course,
can always be done by making the $S^1$-action independent of the
collar parameter. In the presence of additional geometric structures,
the existence of such an extension becomes an honest restriction.

\begin{ex}
\label{ex:S3-cut}
Consider a solid torus $V=S^1\times D^2$ with $S^1$-action on
the boundary $\partial V$ given by the flow of~$\partial_s$,
where $s$ denotes the $S^1$-coordinate. Then the quotient of $V$
with respect to this action on the boundary is the $3$-sphere:
the quotient map can be written explicitly as
\[ \begin{array}{ccc}
S^1\times D^2 & \longrightarrow & S^3\subset\C^2\\
(s;r,\theta)  & \longmapsto     & \bigl(\sqrt{1-r^2}\,\rme^{\rmi s},
                                  r\rme^{\rmi\theta}\bigr).
\end{array}\]
Equivalently, this quotient may be regarded as the cut with respect
to the extended $S^1$-action on a collar of $\partial V\subset V$.
\end{ex}
\subsection{Contact reduction}
Suppose $M$ is a manifold admitting a contact form $\alpha$
and a \emph{strict} contact $S^1$-action, that is, an action
preserving this contact form. If $X$ is the vector field that generates
the $S^1$-action, this translates into $L_X\alpha=0$. The
\emph{momentum map} of the action is defined by
\[ \begin{array}{rccc}
\mu\co & M & \longrightarrow & \R\\
       & p & \longmapsto     & \alpha_p(X_p).
\end{array}\]
In other words, the momentum map of the $S^1$-action
is simply the Hamiltonian function corresponding to the vector field
generating the action.

We compute
\begin{equation*}
\tag{*}
\label{eqn:mu}
\rmd\mu=\rmd(\alpha(X))=L_X\alpha-i_X\rmd\alpha=-i_X\rmd\alpha.
\end{equation*}
This has the following consequences:
\begin{enumerate}[widest=(ii)]
\item[(i)] $X$ is tangent to the levels of~$\mu$;
\item[(ii)] $0$ is a regular value of $\mu$ if and only if $X$ is
non-singular on $\mu^{-1}(0)$.
\end{enumerate}
For (ii), note that the level $\mu^{-1}(0)$ is precisely defined by
the condition $\alpha(X)=0$, and $\rmd\alpha$ is a non-degenerate $2$-form
on $\ker\alpha$. So if $X$ is non-singular, we can find $X'\in\ker\alpha$
such that
\[ \rmd\mu(X')=-\rmd\alpha(X,X')\neq 0.\]
In particular, we notice that if one of the equivalent conditions in
(ii) holds, then the contact structure $\ker\alpha$ is transverse
to $\ker(\rmd\mu)$, i.e.\ to the level set $\mu^{-1}(0)$.

\begin{lem}
If the $S^1$-action on $\mu^{-1}(0)$ is free, the $1$-form $\alpha$
induces a contact form on the quotient manifold $\mu^{-1}(0)/S^1$.
\end{lem}

\begin{proof}
The conditions $\alpha(X)=0$ and $L_X\alpha=0$ imply that $\alpha$
descends to a well-defined nowhere zero $1$-form $\oalpha$
on the quotient.
Moreover, the kernel of $\rmd\alpha_p$ on the transverse intersection
\[ T_p\bigl(\mu^{-1}(0)\bigr)\cap\ker\alpha_p\]
is, by~(\ref{eqn:mu}), spanned by~$X$. Thus, on the quotient under
the $S^1$-action, $\rmd\oalpha$ is non-degenerate
on $\ker\oalpha$.
\end{proof}

\begin{rem}
If one makes the weaker assumption that $0$ is a regular level of~$\mu$,
the $S^1$-action on $\mu^{-1}(0)$ will only be semi-free, in general,
and the quotient a contact \emph{orbifold}.
\end{rem}

An introduction to contact reduction in the case of actions
by arbitrary compact Lie groups can be found in~\cite[Section~7.7]{geig08}.
\subsection{Contact cuts}
\label{subsection:contactcut}
We can now combine the themes of the previous sections. For further
details see \cite{lerm01} or~\cite{agz18b}. Thus, let $M$ be
a manifold carrying a contact form~$\alpha$. Let $\mu_M$ be
the momentum map of a strict contact $S^1$-action on $(M,\alpha)$
generated by the vector field~$X$. As above, we assume that the
$S^1$-action is free on $\mu^{-1}(0)$. We wish to perform the
cut of $M$ at the $0$-level of~$\mu_M$, that is, we want to
collapse the $S^1$-action on the boundary of $\{\mu_M\geq 0\}$, and
find a contact form on this cut.

On the product manifold $M\times\C$ we have the contact form
$\alpha+x\,\rmd y-y\,\rmd x$. The vector field $X-(x\partial_y-y\partial_x)$
generates a strict contact $S^1$-action with momentum map
\[ \mu(p,z)=\mu_M(p)-|z|^2.\]
Then the cut is the reduced manifold $\mu^{-1}(0)/S^1$.

Write $\pi\co\mu^{-1}(0)\rightarrow \mu^{-1}(0)/S^1$ for the quotient map.
The contact form $\oalpha$ on this quotient is characterised by
\[ \pi^*\oalpha=(\alpha+x\,\rmd y-y\,\rmd x)|_{T(\mu^{-1}(0))}.\]
This entails the following.

(i) The composition
\[ \begin{array}{ccccc}
\{p\in M\co\mu_M(p)>0\} & \longrightarrow & \mu^{-1}(0)         &
     \longrightarrow & \mu^{-1}(0)/S^1\\
p                       & \longmapsto     & \bigl(p,\sqrt{\mu_M(p)}\bigr) &
     \longmapsto     & \bigl[\bigl(p,\sqrt{\mu_M(p)}\bigr)\bigr]
\end{array}\]
is an equidimensional strict contact embedding, i.e.\ it pulls back
$\oalpha$ to~$\alpha$.

(ii) The inclusion
\[ \begin{array}{ccc}
\mu_M^{-1}(0) & \longrightarrow & \mu^{-1}(0)\\
p             & \longmapsto     & (p,0)
\end{array}\]
induces a codimension~$2$ strict contact embedding
\[ \mu_M^{-1}(0)/S^1\longrightarrow\mu^{-1}(0)/S^1.\]

The fact that these are \emph{strict} contact embeddings makes
contact cuts particularly useful for controlling the Reeb dynamics,
as we shall see below. For other applications, see~\cite{lerm01}.
\section{Finite numbers of periodic orbits}
\label{section:finite}
My aim in this section is to outline a proof of the
following result from~\cite{agz18a}, which contrasts with
the result of Cristofaro-Gardiner, Hutchings and Pomerleano~\cite{chp17}
mentioned in Section~\ref{subsection:WC}.

\begin{thm}[Albers--Geiges--Zehmisch]
\label{thm:agz1}
In any odd dimension $\geq 5$ there are closed, connected contact manifolds
with arbitrarily large finite numbers of periodic Reeb orbits.
In dimension five, any number $\geq 3$ can be so realised.
\end{thm}

The main tool will be the lifting of a Hamiltonian $S^1$-action
on an integral symplectic manifold $(B,\omega)$ to a contact action
on the Boothby--Wang $S^1$-bundle over $B$ with Euler class the integral
cohomology class $-[\omega/2\pi]$.
\subsection{Lifting Hamiltonian vector fields in Boothby--Wang bundles}
Suppose $(B,\omega)$ is a closed symplectic manifold such that the
de Rham cohomology class $e:=-[\omega/2\pi]$ is integral.
Then one can find a connection $1$-form $\alpha$ on the
$S^1$-bundle $M$ over $B$ of Euler class~$e$ such that $\omega$
is the connection form of~$\alpha$. This means that, writing
$\pi\co M\rightarrow B$ for the bundle projection, we have $\pi^*\omega
=\rmd\alpha$. The $2$-form $\omega$ being symplectic then translates
into $\alpha$ being a contact form.
Contact manifolds of this type are known as Boothby--Wang bundles
or prequantisation bundles, see~\cite[Section~7.2]{geig08}
for further details.

Now consider a Hamiltonian function $H\co B\rightarrow\R$ and let $X=X_H$
be the corresponding Hamiltonian vector field. Write $\Xh$ for the
horizontal lift of $X$ to~$M$, i.e.\ the vector field on $M$
satisfying
\[ \alpha(\Xh)=0,\;\;\; T\pi(\Xh)=X.\]
Also, write $\tH:=H\circ\pi$ for the lift of~$H$, and $R$ for the
Reeb vector field of~$\alpha$.

\begin{lem}
The vector field $\tX:=\tH R+\Xh$ is a strict contact
vector field for the contact form~$\alpha$.
\end{lem}

\begin{proof}
We compute
\[ L_{\tX}\alpha=\rmd(\alpha(\tX))+i_{\tX}\rmd\alpha=
\rmd\tH+i_{\Xh}\rmd\alpha=
\rmd\tH+\pi^*(i_X\omega)=0.\qedhere\]
\renewcommand{\qed}{}
\end{proof}
\subsection{Lifting Hamiltonian $S^1$-actions}
Now suppose the Hamiltonian vector field $X$ induces an action
on $B$ by the circle $S^1=\R/2\pi\Z$. By adding a suitable
constant to the Hamiltonian function~$H$, we may assume that $H>0$ and
$H(p_0)\in\N$ at some chosen singularity $p_0\in B$ of~$X$.

\begin{prop}
The lifted vector field $\tX$ then defines an $S^1$-action on~$M$.
\end{prop}

\begin{proof}
Given a point $p\in B$, let $\gamma$ be a simple path from $p_0$ to~$p$.
Under the $S^1$-action, this path sweeps out a disc~$\Delta$ (in the sense
of smooth singular homology theory), bounded by the $S^1$-orbit
$\beta$ through~$p$, see Figure~\ref{figure:discB}.
This orbit may well be multiply covered
or a fixed point, in which case we have a smooth singular $2$-sphere.

\begin{figure}[h]
\labellist
\small\hair 2pt
\pinlabel $\dot{\gamma}$ [tr] at 337 29
\pinlabel $X$ [tl] at 357 96 
\pinlabel $\Delta$ at 255 215
\pinlabel $p_0$ [r] at 185 153
\pinlabel $p$ [l] at 322 65
\pinlabel $\beta$ [bl] at 327 224
\endlabellist
\centering
\includegraphics[scale=0.45]{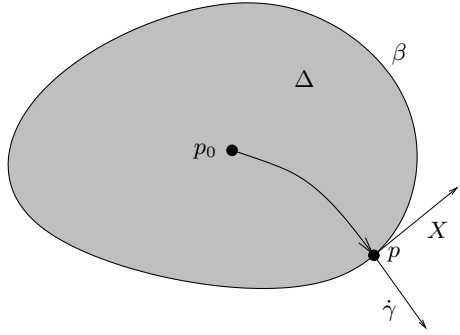}
  \caption{A smooth singular disc in~$B$.}
  \label{figure:discB}
\end{figure}

The horizontal lift $\beta_{\mathrm{h}}$ to~$M$ of $\beta$, starting
at some lift $\tilde{p}$ of~$p$, ends along the $S^1$-fibre
through $\tilde{p}$, with some holonomy shift $h$ mod~$2\pi$.
There is a lifted disc $\tilde{\Delta}$ bounded by
$\beta_{\mathrm{h}}$ and a segment of length $-h$ along
the $S^1$-fibre, see Figure~\ref{figure:discM}.

\begin{figure}[h]
\labellist
\small\hair 2pt
\pinlabel $\text{$S^1$-fibre}$ [l] at 122 415
\pinlabel $\tilde{p}$ [tl] at 167 89
\pinlabel $-h$ [r] at 155 130
\pinlabel $\tilde{\Delta}$ at 278 247
\pinlabel $\beta_{\mathrm{h}}$ [bl] at 322 292
\endlabellist
\centering
\includegraphics[scale=0.45]{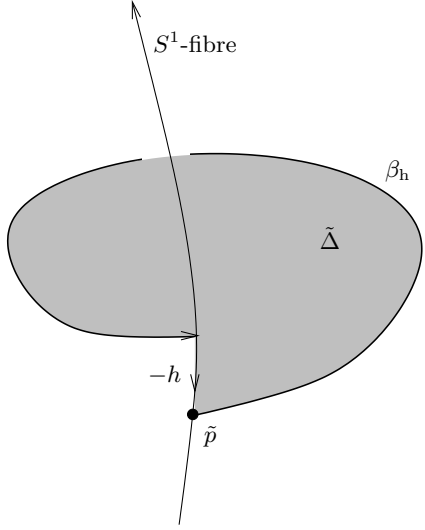}
  \caption{The lifted disc in~$M$.}
  \label{figure:discM}
\end{figure}

We then compute
\begin{eqnarray*}
\lefteqn{-h =\int_{\beta_{\mathrm{h}}}\alpha=\int_{\tilde{\Delta}}\rmd\alpha=
\int_{\Delta}\omega=-2\pi\int_{\gamma}i_X\omega=}\\
 & & 2\pi\int_{\gamma}\rmd H=2\pi\bigl(H(p)-H(p_0)\bigr)
=2\pi H(p)\;\;\text{mod $2\pi$}.
\end{eqnarray*}

Along the $S^1$-orbit $\beta$ of the Hamiltonian flow, the
value of $H$ is constant equal to~$H(p)$. It follows that
the $\tX$-orbit through $\tilde{p}$ closes up after time~$2\pi$.
\end{proof}

Of course the $S^1$-orbit of $\tX$ through $\tilde{p}$ may
be multiply covered. Observe that if $p$ is a fixed point
of the Hamiltonian $S^1$-action on~$B$, then $h=0$, so
our calculation shows that $H(p)$ is also a natural number.
So the $\tX$-orbit starting at a lift of such a fixed point
is simply a fibre of the Boothby--Wang bundle, usually
multiply covered.

\begin{rem}
This proposition is only the most simple case of much more general results
about the lifting of group actions to Boothby--Wang bundles
or prequantisation line bundles; see~\cite{mund01}, for instance.
\end{rem}

\begin{lem}
The flows of $\tX$ and $R$ commute and hence define a strict contact
$T^2$-action.
\end{lem}

\begin{proof}
The function $\tH$ on $M$ is $R$-invariant by construction.
The defining equations for $\Xh$ can be written as $\alpha(\Xh)=0$
and $i_{\Xh}\rmd\alpha=-\rmd\tH$. Since the forms $\alpha$ and
$\rmd\alpha$ are $R$-invariant, so are these defining equations.
It follows that $\tX$ is $R$-invariant, i.e.\ $[R,\Xh]=L_R\Xh=0$.
\end{proof}

It follows that if the $S^1$-action on $B$ has finitely many fixed
points, the flow of $\tX+\varepsilon R$, for $\varepsilon\in
\R^+\setminus\Q$, has finitely many periodic orbits, viz., the
fibres of $M\rightarrow B$ over the fixed points.
By Example~\ref{ex:rescale}, this vector field $\tX+\varepsilon R$
is the Reeb vector field of a suitably rescaled~$\alpha$.
\subsection{Hamiltonian $S^1$-actions with finitely many fixed points}
\label{section:Ham-finite}
In order to prove Theorem~\ref{thm:agz1}, all that remains to be done
is exhibit Hamiltonian $S^1$-actions with large finite numbers of fixed
points (and any number $\geq 3$ in dimension four).
On $\CP^n$ with its standard Fubini--Study symplectic form, we
consider the Hamiltonian function
\[ H\co [z_0:\ldots:z_n]\longmapsto\frac{1}{2}\,\frac{w_1|z_1|^2+
\cdots+w_n|z_n|^2}{|z_0|^2+\cdots+|z_n|^2},\]
where $w_1,\ldots,w_n$ are pairwise distinct integers of greatest
common divisor~$1$. This generates the effective Hamiltonian
$S^1$-action
\[ \rme^{\rmi\varphi}[z_0:\ldots:z_n]=\bigl[z_0:\rme^{\rmi w_1\varphi}:
\ldots:\rme^{\rmi w_n\varphi}\bigr] \]
with precisely $n+1$ fixed points
\[ [1:0:\ldots:0],\ldots,[0:\ldots:0:1].\]

\begin{rem}
A Hamiltonian function defining a Hamiltonian $S^1$-action
with isolated fixed points on some closed symplectic manifold
$(W,\omega)$ is always a (perfect) Morse
function~\cite[Section~32]{gust84}. Since $[\omega^k]$
is a non-trivial generator of $H^{2k}(W)$ for $0\leq k\leq n$,
the minimal number of fixed points of a Hamiltonian $S^1$-action equals
$n+1$. Thus, the construction we described never yields less than
$n+1$ periodic Reeb orbits on a contact manifold of dimension $2n+1$.
This may or may not be regarded as evidence for the conjecture
mentioned in Section~\ref{subsection:WC}.
\end{rem}

The $S^1$-action we described extends to a Hamiltonian $S^1$-action on
the $a$-fold symplectic blow-up $\CP^n\#_a\overline{\CP}^n$ with
$n+1+a(n-1)$ fixed points.
\section{Embedding surface diffeomorphisms into Reeb flows}
\label{section:embedding}
In this section I describe how contact cuts can be used to
construct Reeb flows with a global surface of section
and a given area-preserving diffeomorphism of this surface
as Poincar\'e return map.
\subsection{Global surfaces of section}
Let $X$ be a non-singular vector field on a $3$-manifold $M$.
A \emph{global surface of section} for the flow of $X$,
see Figure~\ref{figure:surf-section}, is
an embedded compact surface $\Sigma$ with boundary such that
\begin{enumerate}[widest=(iii)]
\item[(i)] the boundary $\partial\Sigma$ is a union of orbits;
\item[(ii)] the interior $\Int(\Sigma)$ is transverse to~$X$;
\item[(iii)] the orbit of $X$ through any point in $M\setminus\partial\Sigma$
intersects $\Int(\Sigma)$ in forward and backward time.
\end{enumerate}

\begin{figure}[h]
\centering
\labellist
\small\hair 2pt
\pinlabel $p$ [r] at 214 196
\pinlabel $\psi(p)$ [l] at 226 269
\endlabellist
\includegraphics[scale=0.4]{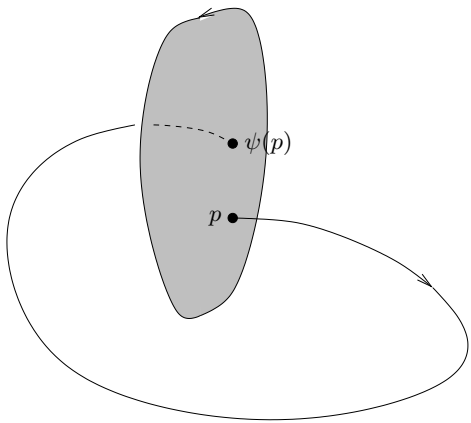}
  \caption{A disc-like global surface of section.}
  \label{figure:surf-section}
\end{figure}

In general, such a surface of section does not exist. For instance,
the aperiodic flow on $S^3$ constructed by Kuperberg does not admit
a global surface of section: this surface would have to be closed
and orientable (thanks to the transverse vector field), but such a surface
separates $S^3$ and cannot satisfy~(iii).

Global surfaces of section were introduced by Poincar\'e in his study
of the $3$-body problem. Given a surface of section~$\Sigma$,
the dynamics of the system is reduced to studying the return map on
the surface, i.e.\ the map $\psi$ that sends each point $p\in\Sigma$
to the first intersection point $\psi(p)\in\Sigma$ of the
orbit through $p$ in forward time.
For instance, closed orbits of the original system correspond
to periodic points of the discrete dynamics on the surface.

Hofer, Wysocki and Zehnder~\cite{hwz98,hwz03} developed holomorphic
curves techniques for finding global surfaces of sections for
Hamiltonian and Reeb flows.
An application is a new proof of the existence of infinitely many
closed geodesics for any Riemannian metric on~$S^2$.

Even for Reeb flows, however, the existence of a surface of section
is not guaranteed. Hryniewicz, Momin and Salom\~ao~\cite{hms15}
describe an example of a Reeb flow on $S^3$ with a Reeb Hopf link,
neither component of which spans a global surface of section.

\begin{ex}
The flow of the Reeb vector field $\partial_{\varphi_1}+(1+\varepsilon)
\partial_{\varphi_2}$ on $S^3$ described in Example~\ref{ex:xist}
has a global surface of section: the disc made up of the
closed right half-plane and the point at infinity.

If we think of $\varphi_1$ as the angle about the vertical axis,
and of the $\varphi_2$-direction as the one along the axis
or along the circles representing the Hopf tori, the return map
on the disc becomes a rotation through an angle~$2\pi\varepsilon$.
\end{ex}
\subsection{Pseudorotations}
\begin{defn}
An \emph{irrational pseudorotation} is a diffeomorphism $\psi$ of $D^2$
such that
\begin{enumerate}[widest=(ii)]
\item[(i)] $\psi$ is area-preserving;
\item[(ii)] $\psi$ has $0\in D^2$ as a fixed point and no other
periodic point.
\end{enumerate}
\end{defn}

Fayad and Katok~\cite{faka04} constructed such pseudorotations as
$C^{\infty}$-limits
\[ \lim_{\nu\rightarrow\infty}\phi_{\nu}\circ\frakR_{p_{\nu}/q_{\nu}}\circ
\phi_{\nu}^{-1},\]
where the $\frakR_{p_{\nu}/q_{\nu}}$ are $2\pi$-rational rotations of
$D^2$ approximating an irrational rotation, and the conjugating
maps $\phi_{\nu}$ are area-preserving diffeomorphisms of~$D^2$,
equal to the identity on a small and, for $\nu\rightarrow\infty$,
shrinking neighbourhood of~$\partial D^2$.

The complexity of these Fayad--Katok pseudorotations is expressed
by the fact that they only admit three ergodic invariant measures:
the Lebesgue measure on the disc or the boundary, and the
$\delta$-measure at~$0$. Even so, there are many non-dense orbits,
so these pseudorotations (and the Reeb flows we are going to
construct with their help) are not minimal. On the issue
of minimality of Hamiltonian and Reeb flows see \cite{fish15}
and~\cite{geze18}.
\subsection{Embedding into Reeb flows}
The theorem I would like to advertise here is the following.
In fact, the main result in~\cite{agz18b} is quite a bit more general.

\begin{thm}[Albers--Geiges--Zehmisch]
\label{thm:agz2}
There is a contact form on $S^3$, inducing the standard contact
structure, whose Reeb flow has a disc-like surface of section on
which the return map equals a given Fayad--Katok pseudorotation.
\end{thm}

Given a diffeomorphism $\psi\co D^2\rightarrow D^2$, one can write
it as the time-$2\pi$ map of a $2\pi$-periodic Hamiltonian
$H_s\co D^2\rightarrow\R$, $s\in\R/2\pi\Z$. In other words,
$\psi$ is the time-$2\pi$ map of the flow of~$X_s$, the
time-dependent Hamiltonian vector field defined by~$H_s$.
As symplectic form on $D^2$ we take $\omega=2r\,\rmd r\wedge\rmd\varphi$.
Then
\[ \alpha:=H_s\,\rmd s+r^2\,\rmd\varphi\]
is a contact form for $H_s$ sufficiently large. Since we can
always add a constant to $H_s$ without changing $X_s$, this
condition on $H_s$ is no restriction. The Reeb vector field
of $\alpha$ is proportional to $\partial_s+X_s$. This means that we have
found a contact form on $S^1\times D^2$ whose Reeb flow is transverse
to the $D^2$-factor, with return map the given~$\psi$.
This observation has been used previously in~\cite{abhs18}.

Theorem~\ref{thm:agz2} now follows by realising $S^3$ as a suitable
contact cut of $(S^1\times D^2,\alpha)$, cf.~Example~\ref{ex:S3-cut}.
Observe that the $\phi_{\nu}\circ\frakR_{p_{\nu}/q_{\nu}}\circ
\phi_{\nu}^{-1}$ are rigid rotations near the boundary of $D^2$,
which implies that there is a strict contact $S^1$-action near
the boundary $\partial(S^1\times D^2)$, tangent to $\ker\alpha$. This allows
one to perform a contact cut, and the observation about strict
contact embeddings we made in Section~\ref{subsection:contactcut}
gives one control over the Reeb dynamics on~$S^3$.

The crucial step, then, is to show that one may pass to the limit
$\nu\rightarrow\infty$ in this construction or, more generally,
to describe a boundary behaviour of $\psi$ that still allows one
to perform a boundary quotient on $S^1\times D^2$. This second approach
allows one to extend the construction to surfaces other than~$D^2$.
\begin{ack}
I would like to thank Gianluca Bande and Beniamino Cap\-pelletti-Montano
for the invitation to present a mini-course of five lectures
during the Cagliari conference.
The perceptive comments of the audience have helped to improve the
presentation in these lecture notes. I thank Sebastian Durst for
comments on a draft version of these notes.
The author is partially supported by the SFB/TRR 191 `Symplectic Structures
in Geometry, Algebra and Dynamics', funded by the Deutsche
Forschungsgemeinschaft.
\end{ack}


\begin{thebibliography}{10}
%
\bibitem{abhs18}
A. Abbondandolo, B. Bramham, U. L. Hryniewicz and P. A. S. Salom\~ao,
Sharp systolic inequalities for Reeb flows on the $3$-sphere,
\textit{Invent. Math.}
\textbf{211} (2018), 687--778.
%
\bibitem{agz18a}
P. Albers, H. Geiges and K. Zehmisch,
Reeb dynamics inspired by Katok's example in Finsler geometry,
\textit{Math. Ann.}
\textbf{370} (2018), 1883--1907.
%
\bibitem{agz18b}
P. Albers, H. Geiges and K. Zehmisch,
Pseudorotations of the $2$-disc and Reeb flows on the $3$-sphere,
\texttt{arXiv:1804.07129}.
%
\bibitem{blai10}
D. Blair,
\textit{Riemannian Geometry of Contact and Symplectic Manifolds}, 2nd edition,
Progr. Math. \textbf{203},
Birkh\"auser Verlag, Basel (2010).
%
\bibitem{coho05}
V. Colin and K. Honda,
Constructions contr\^ol\'ees de champs de Reeb et applications,
\textit{Geom. Topol.}
\textbf{9} (2005), 2193--2226.
%
\bibitem{cghu16}
D. Cristofaro-Gardiner and M. Hutchings,
From one Reeb orbit to two,
\textit{J. Differential Geom.}
\textbf{102} (2016), 25--36.
%
\bibitem{chp17}
D. Cristofaro-Gardiner, M. Hutchings and D. Pomerleano,
Torsion contact forms in three dimensions have two or
infinitely many Reeb orbits,
\texttt{arXiv:1701.02262}.
%
\bibitem{dgz17}
M. D\"orner, H. Geiges and K. Zehmisch,
Finsler geodesics, periodic Reeb orbits, and open books,
\textit{Eur. J. Math.}
\textbf{3} (2017), 1058--1075.
%
\bibitem{elia89}
Ya. Eliashberg,
Classification of overtwisted contact structures on $3$-manifolds,
\textit{Invent. Math.}
\textbf{98} (1989), 623--637.
%
\bibitem{elia92}
Ya. Eliashberg,
Contact $3$-manifolds twenty years since J.~Martinet's work,
\textit{Ann. Inst. Fourier (Grenoble)}\/
\textbf{42} (1992), 165--192.
%
\bibitem{elho94}
Ya. Eliashberg and H. Hofer,
A Hamiltonian characterization of the three-ball,
\textit{Differential Integral Equations}
\textbf{7} (1994), 1303--1324.
%
\bibitem{faka04}
B. Fayad and A. Katok,
Constructions in elliptic dynamics,
\textit{Ergodic Theory Dynam. Systems}
\textbf{24} (2004), 1477--1520.
%
\bibitem{fish15}
J. W. Fish,
Feral pseudoholomorphic curves and minimal sets,
\textit{Oberwolfach Rep.}
\textbf{12} (2015), 1941.
%
\bibitem{fls15}
U. Frauenfelder, C. Labrousse and F. Schlenk,
Slow volume growth for Reeb flows on spherizations and contact Bott--Samelson
theorems,
\textit{J. Topol. Anal.}
\textbf{7} (2015), 407--451.
%
\bibitem{geig08}
H. Geiges,
\textit{An Introduction to Contact Topology},
Cambridge Stud. Adv. Math. \textbf{109},
Cambridge University Press (2008).
%
\bibitem{geig16}
H. Geiges,
\textit{The Geometry of Celestial Mechanics},
London Math. Soc. Stud. Texts \textbf{83},
Cambridge University Press (2016).
%
\bibitem{gepa07}
H. Geiges and F. Pasquotto,
A formula for the Chern classes of symplectic blow-ups,
\textit{J. Lond. Math. Soc.} (2)
\textbf{76} (2007), 313--330.
%
\bibitem{grz14}
H. Geiges, N. R\"ottgen and K. Zehmisch,
Trapped Reeb orbits do not imply periodic ones,
\textit{Invent. Math.}
\textbf{198} (2014), 211--217.
%
\bibitem{grz16}
H. Geiges, N. R\"ottgen and K. Zehmisch,
From a Reeb orbit trap to a Hamiltonian plug,
\textit{Arch. Math. (Basel)}\/
\textbf{107} (2016), 397--404.
%
\bibitem{geze13}
H. Geiges and K. Zehmisch,
How to recognize a $4$-ball when you see one,
\textit{M\"unster J. Math.}
\textbf{6} (2013), 525--554; erratum: pp.~555--556.
%
\bibitem{geze18}
H. Geiges and K. Zehmisch,
Odd-symplectic forms via surgery and minimality in symplectic dynamics,
\textit{Ergodic Theory Dynam. Systems},
to appear.
%
\bibitem{ginz97}
V. L. Ginzburg,
A smooth counterexample to the Hamiltonian Seifert conjecture in~$\R^6$,
\textit{Internat. Math. Res. Notices}\/
\textbf{1997}, 641--650.
%
\bibitem{gigu03}
V. L. Ginzburg and B. Z. G\"urel,
A $C^2$-smooth counterexample to the Hamiltonian Seifert conjecture
in~$\R^4$,
\textit{Ann. of Math.} (2)
\textbf{158} (2003), 953--976.
%
\bibitem{gust84}
V. Guillemin and S. Sternberg,
\textit{Symplectic Techniques in Physics},
Cambridge University Press (1984).
%
\bibitem{herm99}
M. R. Herman,
Examples of compact hypersurfaces in $\R^{2p}$, $2p\geq 6$, with
no periodic orbits,
in \textit{Hamiltonian systems with three or more degrees of
freedom}, NATO Adv. Sci. Inst. Ser. C Math. Phys. Sci. \textbf{553},
Kluwer, Dordrecht (1999), p.~126.
%
\bibitem{hofe93}
H. Hofer,
Pseudoholomorphic curves in symplectizations with applications to the
Weinstein conjecture in dimension three,
\textit{Invent. Math.}
\textbf{114} (1993), 515--563.
%
\bibitem{hovi88}
H. Hofer and C. Viterbo,
The Weinstein conjecture in cotangent bundles and related results,
\textit{Ann. Scuola Norm. Sup. Pisa Cl. Sci.} (4)
\textbf{15} (1988), 411--445.
%
\bibitem{hwz98}
H. Hofer, K. Wysocki and E. Zehnder,
The dynamics on three-dimensional strictly convex energy surfaces,
\textit{Ann. of Math.} (2)
\textbf{148} (1998), 197--289.
%
\bibitem{hwz03}
H. Hofer, K. Wysocki and E. Zehnder,
Finite energy foliations of tight three-spheres and Hamiltonian dynamics,
\textit{Ann. of Math.} (2)
\textbf{157} (2003), 125--255.
%
\bibitem{hms15}
U. Hryniewicz, A. Momin and P. A. S. Salom\~ao,
A Poincar\'e--Birkhoff theorem for tight Reeb flows on~$S^3$,
\textit{Invent. Math.}
\textbf{199} (2015), 333--422.
%
\bibitem{kerm02}
E. Kerman,
New smooth counterexamples to the Hamiltonian Seifert conjecture,
\textit{J. Symplectic Geom.}
\textbf{1} (2002), 253--267.
%
\bibitem{kupe-g96}
G. Kuperberg,
A volume-preserving counterexample to the Seifert conjecture,
\textit{Comment. Math. Helv.}
\textbf{71} (1996), 70--97.
%
\bibitem{kupe-k94}
K. Kuperberg,
A smooth counterexample to the Seifert conjecture,
\textit{Ann. of Math.} (2)
\textbf{140} (1994), 723--732.
%
\bibitem{kupe-k99}
K. Kuperberg,
Aperiodic dynamical systems,
\textit{Notices Amer. Math. Soc.}
\textbf{46} (1999), 1035--1040.
%
\bibitem{lerm95}
E. Lerman,
Symplectic cuts,
\textit{Math. Res. Lett.}
\textbf{2} (1995), 247--258.
%
\bibitem{lerm01}
E. Lerman,
Contact cuts,
\textit{Israel J. Math.}
\textbf{124} (2001), 77--92.
%
\bibitem{lyfe51}
L. A. Lyusternik and A. I. Fet,
Variational problems on closed manifolds,
\textit{Doklady Akad. Nauk SSSR (N.S.)}\/
\textbf{81} (1951), 17--18.
%
\bibitem{mcdu84}
D. McDuff,
Examples of simply-connected symplectic non-K\"ahlerian manifolds,
\textit{J. Differential Geom.}
\textbf{20} (1984), 267--277.
%
\bibitem{mcsa17}
D. McDuff and D. Salamon,
\textit{Introduction to Symplectic Topology}, 3rd edition,
Oxf. Grad. Texts Math. \textbf{27},
Oxford University Press (2017).
%
\bibitem{mund01}
I. Mundet i Riera,
Lifts of smooth group actions to line bundles,
\textit{Bull. London Math. Soc.}
\textbf{33} (2001), 351--361.
%
\bibitem{pame82}
J. Palis, Jr. and W. de Melo,
\textit{Geometric Theory of Dynamical Systems},
Springer-Verlag, Berlin (1982).
%
\bibitem{rabi79}
P. H. Rabinowitz,
Periodic solutions of a Hamiltonian system on a prescribed energy surface,
\textit{J. Differential Equations}\/
\textbf{33} (1979), 336--352.
%
\bibitem{ruki95}
P. Rukimbira,
Topology and closed characteristics of $K$-contact manifolds,
\textit{Bull. Belg. Math. Soc. Simon Stevin}\/
\textbf{2} (1995), 349--356.
%
\bibitem{ruki99}
P. Rukimbira,
On $K$-contact manifolds with minimal number of closed characteristics,
\textit{Proc. Amer. Math. Soc.}
\textbf{127} (1999), 3345--3351.
%
\bibitem{schw74}
P. A. Schweitzer,
Counterexamples to the Seifert conjecture and opening of
closed leaves of foliations,
\textit{Ann. of Math.} (2)
\textbf{100} (1974), 386--400.
%
\bibitem{seif50}
H. Seifert,
Closed integral curves in $3$-space and isotopic two-dimensional
deformations,
\textit{Proc. Amer. Math. Soc.}
\textbf{1} (1950), 287--302.
%
\bibitem{taub07}
C. H. Taubes,
The Seiberg-Witten equations and the Weinstein conjecture,
\textit{Geom. Topol.}
\textbf{11} (2007), 2117--2202.
%
\bibitem{wein79}
A. Weinstein,
On the hypotheses of Rabinowitz' periodic orbit theorems,
\textit{J. Differential Equations}\/
\textbf{33} (1979), 336--352.
%
\bibitem{wils66}
F. W. Wilson, Jr.,
On the minimal sets of non-singular vector fields,
\textit{Ann. of Math.} (2)
\textbf{84} (1966), 529--536.
%
\end{thebibliography}
\end{document}